\documentclass{birkjour}

\usepackage{graphicx} 

\usepackage{amsthm}
\usepackage{amsmath}
\usepackage{amssymb}
\usepackage{xcolor}
\usepackage{subcaption}

\usepackage{tikz}
\usepackage{float}

\newtheorem{proposition}{Proposition}
\newtheorem{theorem}{Theorem}
\newtheorem{lemma}{Lemma}
\newtheorem{corollary}{Corollary}

\title{Machine learning discovers invariants of braids and flat braids}

\author[A.~Lisitsa]{Alexei Lisitsa}
\address{University of Liverpool, UK}
\email{lisitsa@liverpool.ac.uk}

\author[M.~Salles]{Mateo Salles}
\address{University of Essex, UK}
\email{m.salles@essex.ac.uk}

\author[A.~Vernitski]{Alexei Vernitski}
\address{University of Essex, UK}
\email{asvern@essex.ac.uk}

\begin{document}

\maketitle

\begin{abstract}
    We use machine learning to classify examples of braids (or flat braids) as trivial or non-trivial. Our ML takes form of supervised learning using neural networks (multilayer perceptrons). When they achieve good results in classification, we are able to interpret their structure as mathematical conjectures and then prove these conjectures as theorems. As a result, we find new convenient invariants of braids, including a complete invariant of flat braids. 
\end{abstract}

\section{Introduction and AI background} \label{ML}

Automated discovery of theorems or, in other words, AI-assisted conjectures is an increasingly important direction of research in recent years, see, for example, \cite{davies2021advancing, fawzi2022discovering,vasylenko2021element}. In this paper we report on our experiments with applying neural networks to braids and flat braids; this has led to forming conjectures, which we then were able to prove as Theorems \ref{thm:ces_braid_description}, \ref{thm:cep_braid_description}, \ref{thm:ces_permutation_description}.

We are satisfied with the results of this study, and are pleasantly surprised by them. Indeed, normally, our research concentrates on explainable AI; for instance, if we are speaking about a trivial braid, we are speaking of it in the context of being able to find an untangling sequence of Reidemeister moves for this braid \cite{khan2021untangling,lisitsa2023supervised}. In this study we considered the problem of classifying braids as trivial or non-trivial using neural networks; this approach feels woolly and imprecise in comparison with what we normally do. As expected, neural networks were able to produce only a partial solution. Nevertheless, in some of the experiments the classification produced by the neural network was unexpectedly good, and it made us think that there must be a theorem there, and the entries in the trained neural network formed a distinctive pattern which we were able to re-formulate as a theorem. As a result, imprecise experiments with supervised learning and neural networks have led us to proving exact and unambiguous mathematical results. The lessons that we draw from this study are an inspiration to use supervised learning when looking for conjectures and useful skills that will help us to spot potential conjectures.

An artificial neural network is, in the simplest case, a \emph{perceptron}, that is, the process of producing the output from the input by multiplying the input by a matrix (the entries in this matrix are called \emph{weights}) and then applying a monotonic non-linear function (for example, $0$ if $x \le 0$ and $1$ if $x > 0$) to the numbers in the output. If several perceptrons are applied consecutively one after another, this construction is called a \emph{multilayer perceptron (MLP)}; by \emph{layers} one means the input, the output and the \emph{hidden layers}, that is, the intermediate outputs that serve as inputs for other perceptrons. A slightly more general term \emph{feedforward neural network} is also frequently used in practice with the same meaning as MLP. If one wants to stress that the desired behaviour of a neural network cannot be approximated by a neural network with a small number of layers, one speaks of a \emph{deep neural network}. In this paper we do not use deep neural networks; instead, we use multilayer perceptrons with one hidden layer; this gives us an opportunity to inspect the weights and generalize them in the form of a mathematical conjecture.

Let us provide more machine learning details. In this paper we consider a problem of \emph{supervised learning} of the binary classifiers of some properties of braids and flat braids. 
We use  the classical model of multilayer perceptron (MLP) \cite{
perceptron-1961},\cite{rumelhart:errorpropnonote} and its implementation as software called  WEKA Workbench for Data Mining \cite{10.5555/3086818}. 
In terms of machine learning, a MLP is a kind of a feedforward neural network models which supports supervised learning using \emph{backpropagation}  \cite{rumelhart:errorpropnonote}. 
It is known to be an \emph{universal approximator} 
\cite{Cybenko,approx}.
MLP implementation in WEKA supports \emph{sigmoid} activation function \cite{sigmoid}.   In  our  experiments for most of the WEKA settings for MLP model  we have used default values\footnote{See precise settings in the Appendix.} with the exception of the number of neurons in the \emph{single} hidden layer (denoted by $H$ in WEKA); we needed to vary $H$ to make the list of weights of the MLP reasonably small to enable us to interpret the numbers. For each of the reported experiments below, we will present the value of $H$. As the measure of performance, we report \emph{weighted average precision} for all experiments conducted.

\section{Braids: definitions and encodings} \label{sec:encodings}

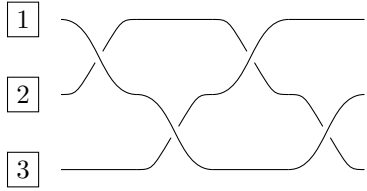
\begin{figure}
\centering
\begin{tikzpicture}[scale=0.50]
\draw (0,0) -- (2,0);
\draw (0,4) .. controls (1,4) and (1,2) .. (2,2);
\draw (0,2) .. controls (0.4,2)  .. (0.9,2.85);
\draw (2,4) .. controls (1.6,4)  .. (1.1,3.15);
\draw (2,2) .. controls (3,2) and (3,0) .. (4,0);
\draw (2,0) .. controls (2.4,0)  .. (2.9,0.85);
\draw (4,2) .. controls (3.6,2)  .. (3.1,1.15);
\draw (2,4) -- (4,4);
\draw (4,0) -- (6,0);
\draw (4,2) .. controls (5,2) and (5,4) .. (6,4);
\draw (4,4) .. controls (4.4,4)  .. (4.9,3.15);
\draw (6,2) .. controls (5.6,2)  .. (5.1,2.85);
\draw (6,0) .. controls (7,0) and (7,2) .. (8,2);
\draw (6,2) .. controls (6.4,2)  .. (6.9,1.15);
\draw (8,0) .. controls (7.6,0)  .. (7.1,0.85);
\draw (6,4) -- (8,4);
\node[draw] at (-1,4) {1};
\node[draw] at (-1,2) {2};
\node[draw] at (-1,0) {3};
\end{tikzpicture}
\caption{An example of a braid to illustrate encodings.}
\label{fig:braid-example}
\end{figure}

This paper concentrates on braids with $3$ strands (unless stated otherwise); an example of such a braid is shown in Figure \ref{fig:braid-example}. We visualize a braid as stretched from the left to the right. Denote positions of strands in the braid by $1, 2, 3$ from the top to the bottom. Recall that the clockwise half-turn swapping the positions of two adjacent strands in positions $i, i+1$ in a braid is denoted by $\sigma_i$ \cite[Section 1.2.4]{kassel2008braid}; thus, the braid in Figure \ref{fig:braid-example} is $\sigma_1 \sigma_2 \sigma_1^{-1} \sigma_2^{-1}$. Denote the number of crossings in a braid or, equivalently, the number of generators $\sigma_i$ featuring in the word describing the braid, by $k$; for example, for the braid in Figure \ref{fig:braid-example} we have $k=4$. In this study we use neural networks to work with braids, therefore, it is convenient to encode braids as matrices. We have considered the following encodings. In each of the encodings, the braid is represented by a matrix containing $k$ columns. Each column in the matrix describes one crossing in the braid, in the same order (from the left to the right) as in the braid. There are exactly $2$ non-zero entries in each column or exactly $1$ non-zero entry in each column; the numbers in the notation for encodings correspond to these values.

\subsubsection*{Encodings based on positions (EP)}

\paragraph{Encoding EP2} The matrix has $3$ rows. Crossings of the type $\sigma_1$, $\sigma_1^{-1}$, $\sigma_2$, $\sigma_2^{-1}$, are denoted by columns $\begin{pmatrix} 1 \\ -1 \\ 0 \end{pmatrix}$, $\begin{pmatrix} -1 \\ 1 \\ 0 \end{pmatrix}$, $\begin{pmatrix} 0 \\ 1 \\ -1\end{pmatrix}$, $\begin{pmatrix} 0 \\ -1 \\ 1\end{pmatrix}$, respectively. Thus, non-zero entries in the matrix indicate the positions of the strands that participate in the crossing and signs of crossings (that is, which strand passes over the crossing, and which strand passes under the crossing). The braid in Figure \ref{fig:braid-example} is encoded as $\begin{pmatrix} 1 & 0 & -1 & 0 \\ -1 & 1 & 1 & -1  \\ 0 & -1 & 0 & 1 \end{pmatrix}$.

\paragraph{Encoding EP1} The matrix has $2$ rows. Crossings of the type $\sigma_1$, $\sigma_1^{-1}$, $\sigma_2$, $\sigma_2^{-1}$, are denoted by columns $\begin{pmatrix} 1 \\ 0 \end{pmatrix}$, $\begin{pmatrix} -1 \\ 0 \end{pmatrix}$, $\begin{pmatrix} 0 \\ 1 \end{pmatrix}$, $\begin{pmatrix} 0 \\ -1 \end{pmatrix}$, respectively. Thus, non-zero entries in the matrix indicate the positions of crossings and their signs. The braid in Figure \ref{fig:braid-example} is encoded as $\begin{pmatrix} 1 & 0 & -1 & 0 \\ 0 & 1 & 0 & -1 \end{pmatrix}$.

\subsubsection*{Encodings based on strands (ES)}

Denote the strands of the braid by $1, 2, 3$, according to their positions on the left-hand side of the braid. For example, strand $1$ in Figure \ref{fig:braid-example} is the curve that begins in position $1$, then moves to position $2$, then to position $3$, and then back to position $2$. 

\paragraph{Encoding ES2} The matrix has $3$ rows, corresponding to strands $1, 2, 3$, respectively. The column describing a crossing contains $1$ in the position corresponding to the strand passing over the crossing and $-1$ in the position corresponding to the strand passing under the crossing. The braid in Figure \ref{fig:braid-example} is encoded as $\begin{pmatrix} 1 & 1 & 0 & 1 \\ -1 & 0 & -1 & -1  \\ 0 & -1 & 1 & 0 \end{pmatrix}$. 

\paragraph{Encoding ES1}  The matrix has $3$ rows; row $1$ (row $2$, row $3$) is used to record crossings of strands $1$ and $2$ (strands $2$ and $3$, strands $3$ and $1$). The entry $1$ (or $-1$) indicates that the former strand passes above (or below) the latter strand. The braid in Figure \ref{fig:braid-example} is encoded as $\begin{pmatrix} 1 & 0 & 0 & 1 \\ 0 & 0 & -1 & 0  \\ 0 & -1 & 0 & 0 \end{pmatrix}$. Note that unlike the other three encodings, which can be easily applied to represent braids with any other number of strands, encoding ES1 does not easily generalize to braids with another number of strands.

\begin{figure}
\centering
\begin{tikzpicture}[scale=0.50]
\draw (0,0) -- (2,0);
\draw (0,2) .. controls (1,2) and (1,4) .. (2,4);
\draw (0,4) .. controls (1,4) and (1,2) .. (2,2);
\draw (2,0) .. controls (3,0) and (3,2) .. (4,2);
\draw (2,2) .. controls (3,2) and (3,0) .. (4,0);
\draw (2,4) -- (4,4);
\draw (4,0) -- (6,0);
\draw (4,2) .. controls (5,2) and (5,4) .. (6,4);
\draw (4,4) .. controls (5,4) and (5,2) .. (6,2);
\draw (6,0) .. controls (7,0) and (7,2) .. (8,2);
\draw (6,2) .. controls (7,2) and (7,0) .. (8,0);
\draw (6,4) -- (8,4);
\node[draw] at (-1,4) {1};
\node[draw] at (-1,2) {2};
\node[draw] at (-1,0) {3};
\end{tikzpicture}
\caption{An example of a flat braid to illustrate encodings.}
\label{fig:flat-braid-example}
\end{figure}
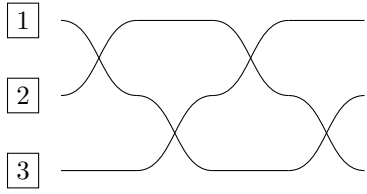

If a permutation is represented as a product of adjacent transpositions, it can be shown using a braid-like diagram called a flat braid. For example, Figure \ref{fig:flat-braid-example} shows a decomposition of permutation $(2 \; 3 \; 1)$ into a product of adjacent transpositions $(1 \; 2)(2 \; 3)(1 \; 2)(2 \; 3)$. All encodings of braids as matrices introduced above can be applied to encoding permutations as matrices, with the only difference that all non-zero entries in the matrices are $1$ instead of either $1$ or $-1$. Thus, for instance, the encoding ES2 of the flat braid in Figure \ref{fig:flat-braid-example} is $\begin{pmatrix} 1 & 1 & 0 & 1 \\ 1 & 0 & 1 & 1  \\ 0 & 1 & 1 & 0 \end{pmatrix}$.

For a given braid, we will denote the matrices encoding this braid in encodings EP2, EP1, ES2, ES1 by $[EP2]$, $[EP1]$, $[ES2]$, $[ES1]$, respectively. 

As to encodings EP2 and EP1, they can be easily transformed into each other; namely, $[EP1] = \begin{pmatrix} 1 & 0 & 0 \\ 0 & 0 & -1  \end{pmatrix} [EP2]$ and $[EP2] = \begin{pmatrix} 1 & 0  \\ -1 & 1 \\ 0 & -1 \end{pmatrix} [EP1]$. Since EP2 and EP1 can be transformed into one another by matrix multiplication, it does not make a difference for an MLP which one to use. Out of EP2 and EP1, we prefer encoding EP1 because it has a slightly smaller matrix (with $2$ rows) and more importantly, using it produces a slightly more readable version of Theorem \ref{thm:cep_braid_description}. Thus, below we use only EP1 and not EP2.

As to encodings ES2 and ES1, it is possible to transform ES1 into ES2 using matrix multiplication, namely, $[ES2] = \begin{pmatrix} 1 & 0 & -1  \\ -1 & 1 & 0 \\ 0 & -1 & 1 \end{pmatrix} [ES1]$. However, it is not possible to transform ES2 into ES1 using matrix multiplication; accordingly, we will see, informally speaking, that ES1 can convey `slightly more information' to the MLP than ES2. We will discuss this difference between ES2 and ES1 in Sections \ref{sec:pure} and \ref{sec:gap-CES2-CES1}. We find ES2 more readable than ES1, since ES2 explicitly refers to specific strands; therefore, when possible, we use ES2; when ES2 is `too weak', we explicitly say that we switch to ES1. 

In computational experiments below, we frequently speak about generating random braids. By this we mean that the number of crossings $k$ is fixed, and then a random braid, in encoding EP1, is produced as a matrix of size $2$ by $k$, in which each column each chosen, at random, out of the following possibilities $\begin{pmatrix} 1 \\ 0 \end{pmatrix}$, $\begin{pmatrix} -1 \\ 0 \end{pmatrix}$, $\begin{pmatrix} 0 \\ 1 \end{pmatrix}$, $\begin{pmatrix} 0 \\ -1 \end{pmatrix}$.

Unlike EP, when one uses ES, one cannot just place randomly chosen columns next to each other to produce a braid, because not every matrix formed of arbitrarily chosen columns permitted in ES2 (or ES1) corresponds to a braid. Let us say that $[ES2]$ (or $[ES1]$) is \emph{realizable} if it corresponds to a braid. For example, $[ES2] = \begin{pmatrix} 1 & 0 \\ 1 & 1 \\ 0 & 1 \end{pmatrix}$ does not correspond to any flat braid, so it is not realizable.

Recall that two braids are equal in the braid group if and only if they can be transformed into each other using two types of operations called the second Reidemeister move (R2) and the third Reidemeister move (R3). (For braids on more than $3$ strands, other types of operations also need to be used in addition to R2 and R3.) Let us remind ourselves what R2 and R3 are. In terms of EP1, R2 means removing from $[EP1]$ or inserting into $[EP1]$ two consecutive columns of the form $\begin{pmatrix} 1 & -1 \\ 0 & 0 \end{pmatrix}$ or $\begin{pmatrix} -1 & 1 \\ 0 & 0 \end{pmatrix}$ or $\begin{pmatrix} 0 & 0 \\ 1 & -1 \end{pmatrix}$ or $\begin{pmatrix} 0 & 0 \\ -1 & 1 \end{pmatrix}$. The move R3 means locating in $[EP1]$ three consecutive columns of the form $\begin{pmatrix} x & 0 & z \\ 0 & y & 0 \end{pmatrix}$ or $\begin{pmatrix} 0 & y & 0 \\ x & 0 & z \end{pmatrix}$ and changing these columns to $\begin{pmatrix} 0 & y & 0 \\ z & 0 & x \end{pmatrix}$ or $\begin{pmatrix} z & 0 & x \\ 0 & y & 0 \end{pmatrix}$, respectively. Note that we are allowed to apply R3 only if $x, y, z$ are not all equal to each other. 

In terms of ES2, R2 means removing from $[ES2]$ or inserting into $[ES2]$ two consecutive columns in which one of the rows is $\begin{pmatrix} 1 & 1 \end{pmatrix}$, another is $\begin{pmatrix} -1 & -1 \end{pmatrix}$, and the third is $\begin{pmatrix} 0 & 0 \end{pmatrix}$. The move R3 means locating in $[ES2]$ three consecutive columns in which one of the rows is $\begin{pmatrix} x & -y & 0 \end{pmatrix}$, another is $\begin{pmatrix} -x & 0 & z \end{pmatrix}$, and the third is $\begin{pmatrix} 0 & y & -z \end{pmatrix}$ and swapping the positions of the first and the third of these columns. Note that we are allowed to apply R3 only if $x, y, z$ are not all equal to each other.

For flat braids, R2 and R3 are defined in the same way, only, obviously, $1$ is used everywhere instead of $-1$, and the condition on $x, y, z$ not being equal to each other is dropped.

A braid is called \emph{trivial} if it is equal to the empty braid (that is, the braid with $k=0$) in the braid group or, equivalently, if it can be transformed into the empty braid using a sequence of applications of R2 and R3. A property of a braid is called an \emph{invariant} if it is preserved by applications of R2 and R3; for example, being trivial is an invariant. 

A flat braid is trivial if and only if the permutation expressed by the flat braid is the identity permutation. There is a natural mapping (actually, a group homomorphism \cite[Section 1.2.4]{kassel2008braid}) from the set of braids to the set of flat braids, acting by ``forgetting'' the signs of the crossings. For example, the flat braid in Figure \ref{fig:flat-braid-example} is the image of the braid in Figure \ref{fig:braid-example}. If the flat braid corresponding to a braid is a trivial flat braid then the braid is called \emph{pure}.

\section{Discovering theorems on braids}

\subsection{Experiments}
In the first series of experiments we considered datasets of random braids and studied learnability 
of the property of braids to be trivial using MLPs in standard setting, as described  in Section \ref{ML}. Note that all datasets we used were balanced, that is, they contained an equal number of trivial and non-trivial braids. 
\begin{table}[hbt!]
\begin{tabular}{||c c c c||} 
 \hline
 H & 50\% /2-fold  & 67\%/3-fold  & 75\% /4-fold\\ [0.5ex] 
 \hline\hline
 1 & 0.826 & 0.735 & 0.709 \\ 
 \hline
 2 & 0.992 & 0.957 & 0.964 \\
 \hline
 3 & 0.993 & 0.993 & 0.992 \\
 \hline
 4 & 0.991 & 0.972 & 0.992 \\
 \hline
 5 & 0.990 & 0.992 & 0.992 \\ 
 \hline
 a=19 & 0.990 & 0.992 & 0.992 \\ 
 \hline
\end{tabular}
\caption{(3,12)-braids triviality recognition. Encoding ES2. Dataset of 8000 braids.    Weighted Avg. Precision for MLP with $H$ neurons in hidden layer and for k\% training set. 2-fold, 3-fold and 4-fold validation applied. "a" denotes default value of H in WEKA: a = [number of attributes + number classes]/2.    
}
\label{table:1}
\end{table}
Table~\ref{table:1} presents the results of one such experiment. Weighted Average Precision is shown for various numbers $H$ of neurons in a single hidden layer. The dataset of 8000 randomly generated braids (4000 trivial and 4000 non-trivial) on 3 threads and of lengths 12 is used for training.  As one can see from the table, MLP can learn the property with very high precision, in excess of $99\%$ even for small hidden layers containing only $H=2$ neurons. We conducted similar experiments in encoding EP1. The property is also learnable with this encoding,  but accuracy was not so high as can be seen from Table~\ref{table:2} 

\begin{table}[hbt!]
\begin{tabular}{||c c c c||} 
 \hline
 H & 50\%/2-fold & 67\%/3-fold & 75\%/4-fold \\ [0.5ex] 
 \hline\hline
 1 & 0.508 & 0.576 & 0.795 \\ 
 \hline
 2 & 0.902 & 0.908 & 0.795 \\
 \hline
 3 & 0.904 & 0.772 & 0.902 \\
 \hline
 4 & 0.898 & 0.908 & 0.904 \\
 \hline
 5 & 0.899 & 0.907 & 0.900 \\ 
 \hline
 a=13 & 0.871
 & 0.889 & 0.890 \\ 
 \hline
\end{tabular}
\caption{(3,12)-braids recognition. Encoding EP1. Dataset of 8000 braids.    Weighted Avg. Precision for MLP with H neurons  in hidden layer and for k\% training set.  3-fold and 4-fold validation applied. 
}
\label{table:2}
\end{table}

Due to the very high accuracy of the MLPs working with encoding ES2, presented in Table \ref{table:1}, we considered the weights of these MLPs trying to find useful mathematical patterns. As we will describe in Section \ref{sec:proofs-for-braids}, we were able to interpret the weights in MLPs classifying braids in encoding ES2 as a conjecture and then prove this conjecture as Theorem \ref{thm:ces_braid_description}. The condition which we extracted from the weights in the MLPs and which is a necessary condition for a braid to be trivial will be denoted by CES2. 

In the second series of experiments we considered datasets of random braids satisfying condition CES2; we will call such datasets CES-datasets. We explored learnability of the property of braids to be trivial for the CES-datasets in encoding EP1. Table~\ref{table:3} presents the results of one such experiment. 

\begin{table}[hbt!]
\begin{tabular}{||c c c c||} 
 \hline
 H & 50\% /2-fold  & 67\%/3-fold  & 75\% /4-fold\\ [0.5ex] 
 \hline\hline
 1 & 0.607 & 0.777 & 0.512 \\ 
 \hline
 2 & 0.721 & 0.899 & 0.688 \\
 \hline
 3 & 0.881 & 0.898 & 0.889 \\
 \hline
 4 & 0.875 & 0.897 & 0.889 \\
 \hline
 5 & 0.859 & 0.892 & 0.854 \\ 
 \hline
 a=13 & 0.894 & 0.886 & 0.889 \\ 
 \hline
\end{tabular}
\caption{(3,12)-braids recognition. Encoding EP1. CES-dataset of 2000 braids.    Weighted Avg. Precision for MLP with H n in hidden layer and for k\% training set. 2-fold, 3-fold and 4-fold validation applied. 
}
\label{table:3}
\end{table}

As we will describe in Section \ref{sec:proofs-for-braids}, we were able to interpret the weights in MLPs classifying braids in CES-datasets in encoding EP1 as a conjecture and then prove this conjecture as Theorem \ref{thm:cep_braid_description}. The condition which we extracted from the weights in the MLPs and which is a necessary condition for a braid to be trivial will be denoted by CEP.

\subsection{Proving the conjectures found by MLPs} \label{sec:proofs-for-braids}

When we considered weights of MLPs in experiments presented in Table \ref{table:1}, we saw that consistently they all have approximately the same absolute value, and their signs, $+$ or $-$, alternate. In this sense, we can say that the computer has formulated a conjecture regarding trivial braids. Below we reformulate this conjecture as Theorem \ref{thm:ces_braid_description} and prove it.

For a number $x$, by $[x]_{m \times n}$ we will denote a matrix of size $m$ by $n$ in which every entry is equal to $x$. By $[\pm]_{m \times n}$ we will denote a matrix of size $m$ by $n$ in which an entry at the position $i, j$ is $1$ (or $-1$) if $i+j$ is even (or odd).

Let $ES$ stand for $[ES2]$ or $[ES1]$. We will say that $[ES]$ satisfies condition CES if the alternating sum of each row is $0$, that is, $[ES][\pm]_{k \times 1} = [0]_{3 \times 1}$. If we need to stress which one of $[ES2]$ or $[ES1]$ is tested, we will use notation CES2 or CES1. 

\begin{theorem} \label{thm:ces_braid_description}
Consider a braid on $3$ strands. Let $ES$ stand for $[ES2]$ or $[ES1]$. \\
1) The list of alternating sums of each row, that is, the column $[ES][\pm]_{k \times 1}$ is a braid invariant. \\
2) If a braid is trivial then condition CES is satisfied.
\end{theorem} 
\begin{proof}
1) Considering how R2 and R3 act on $[ES2]$ (see Section \ref{sec:encodings}), we can see that applying R2 or R3 does not affect the alternating sums of the rows. Note that applying R2 changes the value of $k$, making the braid $2$ crossings shorter or $2$ crossings longer, but the alternating sums of the rows are not affected.

2) If a braid is trivial then it can be transformed into the empty braid using R2 and R3. As we established in 1), applying R2 and R3 does not change the alternating sums of the rows, and the alternating sums for the empty braid are $[0]_{3 \times 1}$. Therefore, the alternating sums of the rows are $[0]_{3 \times 1}$ for each trivial braid.

The proof for $[ES1]$ can be constructed in the same way as 1), 2) above. 
\end{proof}

When we considered weights of MLPs in experiments presented in Table \ref{table:3}, we saw that consistently they all have approximately the same  value. In this sense, the computer has formulated another conjecture regarding trivial braids. Below we reformulate this conjecture as Theorem \ref{thm:cep_braid_description} and prove it.

We will say that $[EP1]$ satisfies condition CEP if the sum of all entries in $[EP1]$ is equal to $0$, that is, $[1]_{1 \times 2}[EP1][1]_{k \times 1} = 0$.

\begin{theorem} \label{thm:cep_braid_description} 
Consider a braid on $3$ strands. 
1) The sum of all entries in $[EP1]$, that is, $[1]_{1 \times 3}[ES2][1]_{k \times 1}$ is a braid invariant. \\
2) If a braid is trivial then condition CEP is satisfied.
\end{theorem} 
\begin{proof}
The proof for $[EP1]$ can be constructed in the same way as the proof for $[ES2]$ in Theorem \ref{thm:ces_braid_description}. 
\end{proof}

\section{Why Theorems \ref{thm:ces_braid_description}, \ref{thm:cep_braid_description} are not ``if and only if''}

Conditions CES in Theorem \ref{thm:ces_braid_description} and CEP in Theorem \ref{thm:cep_braid_description} are necessary conditions for a braid being trivial, and they are also ``almost sufficient'' conditions. To give an example for braid length $k=12$, Theorem \ref{thm:ces_braid_description} correctly classifies $99\%$ of braids, and in addition to that, Theorem \ref{thm:cep_braid_description} correctly classifies $85\%$ of the remaining $1\%$ of braids. 

However, after that, MLPs were not able to generate more conjectures which would complement CES and CEP, producing an “if and only if” result describing trivial braids. One simple explanation which might be offered is that MLPs can only calculate linear functions (or relatively simple generalizations of linear functions), and not every function is linear; in particular, perhaps trivial braids cannot be described by linear functions. However, we can offer a more nuanced information-theoretical explanation for not being able to describe trivial braids using MLPs. 

Indeed, suppose a braid on $3$ strands is represented by a matrix $M$. Assume that like in our encodings, that the number of columns in $M$ is $k$, and the entries in $M$ are in the range $-1, 0, 1$. Suppose the number of rows in $M$ is $3$. Suppose that we can multiply the entries of $M$ by some coefficients (as we do in Theorems \ref{thm:ces_braid_description} and \ref{thm:cep_braid_description}), and these coefficients are in the range $-r, -r+1, \dots, r$, and the value of this matrix product tells us if the braid is trivial. In other words, for some linear function $f$ with coefficients $-c, -c+1, \dots, c$, we have $f(M)=0$ if and only if the braid is trivial. Since only braids with even values of $k$ can be trivial, we will assume that $k$ is even. A rule that can decide if a braid $b$ of length $k$ is trivial can also be used to decide if two braids $c, d$ of length $k/2$ are equal in the braid group (by checking if $cd^{-1}$ is trivial). Here, the entries of the matrix corresponding to $c$ are the first $k/2$ columns (let us denote them by $C$), and the entries of the matrix corresponding to $d$ are the last $k/2$ columns (let us denote them by $D$). If we put separately the terms in $f$ referring to $C$ and to $D$, we can rewrite the equality $f(M)=0$ as $g(C)=h(D)$, for some linear functions $g$ and $h$. Now let us consider the total range of possible values of $g$ (and $h$); it is easy to see that the value of $g$ (or $h$) lies between $-3rk/2$ and $3rk/2$. Assuming that $r$ is fixed and does not depend on the value of $k$ (as it is fixed in conditions CES and CEP), we see that the total range of possible values of $g$ (and $h$) grows linearly with $k$. However, the number of classes of equal braids of length $k/2$ (in other words, the growth rate of the braid group $B_3$, or, equivalently, the growth rate of the fundamental group of the trefoil knot) grows faster than linearly\footnote{A simple explicit example of an exponential-size family of braids of length $k$ which are pairwise distinct in $B_3$ can be produced for even values of $k=2m$ by considering words $\prod_{i=1}^m \sigma_{a(i)}^2$ for $a \in \{1, 2\}^m$.} with $k$. 
This is why for some $k$ there are more classes of equal braids of length $k/2$ than the number of possible values of $g$ (or $h$). Therefore, it is impossible to decide if two braids $c$ and $d$ of length $k/2$ are equal in the braid group by checking the equality $g(C)=h(D)$. Therefore, it is impossible to decide if a braid $b$ on $3$ strands of length $k$ is trivial by checking the equality $f(M)=0$. 

The same argument can also be applied to the scenario of using several linear functions $f_1, \dots, f_n$; thus, we can show that one cannot decide if a braid is trivial by checking a fixed number of equalities $f_1(M)=0, \dots, f_n(M)=0$. Hence follows a fact that we summarized in the beginning of this section, that Theorems \ref{thm:ces_braid_description} and \ref{thm:cep_braid_description} used together cannot correctly describe all trivial braids. Also, from this it follows that if we discover, in addition to CES and CEP, finitely many new conditions of the form $f_1(M)=0, \dots, f_n(M)=0$, all these conditions together will form a necessary condition for a braid to be trivial, but not a sufficient condition. 

If the growth rate of a group is not faster than linear then one can hope to find a rule describing trivial elements using an MLP. For example, in this paper in Section \ref{sec:flat} and Section \ref{sec:2-strands} we demonstrate that MLPs successfully form conjectures which correctly describe trivial flat braids on $3$ strands (the group $\mathfrak{S}_3$ is finite, the growth rate is constant) and trivial braids on $2$ strands (the group is infinite cyclic, the growth rate is linear). 

On a related topic, one might ask what is the smallest example of a braid which satisfies conditions CES and CEP but is not trivial. The answer is that such examples are the braids whose closure is the link known as Borromean rings; one such braid is $\sigma_1 \sigma_2^{-1} \sigma_1 \sigma_2^{-1} \sigma_1 \sigma_2^{-1}$.

\section{Discovering theorems on flat braids} \label{sec:flat}

Our success with Theorems \ref{thm:ces_braid_description} and \ref{thm:cep_braid_description} has inspired us to ask if MLPs can also produce conjectures regarding permutations represented as flat braids. As you will see, for permutations (of 3 elements), Theorem \ref{thm:ces_braid_description} can be reformulated as an “if and only if” result, which exactly describes those permutations which are equal to the identity permutation. 

Tables~\ref{table:4} and \ref{table:5}  report on the results of ML applied to the problem of triviality of flat braids on 3 strands using ES2 and EP1 encodings, respectively. 

\begin{table}[hbt!]
\begin{tabular}{||c c c c||} 
 \hline
 H & 50\% /2-fold  & 67\%/3-fold  & 75\% /4-fold\\ [0.5ex] 
 \hline\hline
 0 & 1.000 & 1.000 & 1.000\\
 \hline
 1 & 1.000 & 1.000 & 1.000 \\ 
 \hline
 2 & 1.000 & 1.000 & 1.000 \\
 \hline
 3 & 1.000 & 1.000 & 1.000 \\
 \hline
 4 & 1.000 & 1.000 & 1.000 \\
 \hline
 5 & 1.000 & 1.000 & 1.000 \\ 
 \hline
 a=19 & 1.000 & 1.000 & 1.000 \\ 
 \hline
\end{tabular}
\caption{(3,12)-flat braids triviality recognition. Encoding ES2. Dataset of 2000
braids.    Weighted Avg. Precision for MLP with $H$ neurons in hidden layer and for k\% training set. 2-fold, 3-fold and 4-fold validation applied. 
}
\label{table:4}
\end{table}

\begin{table}[hbt!]
\begin{tabular}{||c c c c||} 
 \hline
 H & 50\% /2-fold  & 67\%/3-fold  & 75\% /4-fold\\ [0.5ex] 
 \hline\hline
 0 & 0.493 & 0.500 & 0.505\\
 \hline
 1 & ? & 0.595 & 0.544 \\ 
 \hline
 2 & 0.533 & 0.723 & 0.655\\
 \hline
 3 & 0.692 & 0.713 & 0.722\\
 \hline
 4 & 0.645 & 0.882 & 0.742 \\
 \hline
 5 & 0.860 & 0.857 & 0.808 \\ 
 \hline
 a=13 & 0.926 & 0.984 & 0.917 \\ 
 \hline
\end{tabular}
\caption{(3,12)-flat braids triviality recognition. Encoding EP1. Dataset of 2000
braids.    Weighted Avg. Precision for MLP with $H$ neurons in hidden layer and for k\% training set. 2-fold, 3-fold and 4-fold validation applied. 
``?'' denotes an experiment in which  precision for non-trivial braids cannot be computed because they all have been misclassified as trivial.    
}
\label{table:5}
\end{table}

When we considered weights of MLPs in experiments presented in Table \ref{table:4}, we saw that consistently they all have approximately the same absolute value, and their signs, $+$ or $-$, alternate. (The fact that these weights are the same as those produced by MLPs presented in Table \ref{table:1} is an unexpected coincidence, since numbers in ES2 for braids and ES2 for flat braids are not the same; inceed, in ES2 for braids there is one entry equal to $1$ and one entry equal to $-1$ in each column, whereas in ES2 for flat braids there are two entries equal to $1$ in each column.) Thus, the computer has formulated a conjecture regarding trivial braids. Below we reformulate this conjecture as Theorem \ref{thm:ces_permutation_description} and prove it.

The following observation is useful. 

\begin{lemma} \label{lem:flat}
Consider a flat braid on $3$ strands. Suppose a column $c$ features more than once $[ES2]$, say, in positions $i$ and $j$, where $i < j$. Suppose none of columns in positions between $i$ and $j$ is equal to $c$. Then the difference $j-i$ is odd.
\end{lemma}
\begin{proof}
Say, $c = \begin{pmatrix} 1  \\ 1    \\ 0 \end{pmatrix}$. This column represents an intersection of strand $1$ and strand $2$. Strands $1$ and $2$ cross each other at columns with positions $i$ and $j$, but not between them. Each crossing in columns with positions between $i$ and $j$ is either strand $3$ entering the gap between strands $1$ and $2$ (for example, see the second crossing in the flat braid in Figure \ref{fig:flat-braid-example}) or strand $3$ leaving the gap between strands $1$ and $2$  (for example, see the third crossing in the flat braid in Figure \ref{fig:flat-braid-example}). In total, between column $i$ and column $j$, the number of times strand $3$ enters and leaves the gap between strands $1$ and $2$ must be even, to ensure that strands $1$ and $2$ are adjacent to each other both in column $i$ and in column $j$. Hence, $j-i$ is odd.
\end{proof}

\begin{theorem} \label{thm:ces_permutation_description} 
Consider a flat braid on $3$ strands. Let $ES$ stand for $[ES2]$ or $[ES1]$. \\
1) The list of alternating sums of each row, that is, the column $[ES][\pm]_{k \times 1}$ is a flat braid invariant. \\
2) A flat braid is trivial if and only if condition CES is satisfied.
\end{theorem} 
\begin{proof}
The proof is the same for ES2 and ES1; below we use ES2.

Using the same plan of the proof as in Theorems \ref{thm:ces_braid_description}, \ref{thm:cep_braid_description}, one can prove that $[ES2][\pm]_{k \times 1}$ is a braid invariant and in a trivial flat braid, condition CES is satisfied.

Conversely, suppose condition CES is satisfied. Consider a case when a column, say, $c = \begin{pmatrix} 1  \\ 1    \\ 0 \end{pmatrix}$ features more than once in the matrix, say, in positions $i$ and $j$, where $i < j$. We can assume that none of columns in positions between $i$ and $j$ is equal to $c$, and no two columns between $i$ and $j$ are equal to one another; otherwise, change the choice of $i$ and $j$ and $c$, if needed. Since the same two strands (strand 1 and strand 2) cross each other at position $i$ and position $j$, the difference $j-i$ is odd, see Lemma \ref{lem:flat}. If $j-i > 1$, note that columns $i, i+1, i+2$ are pairwise distinct, so we can apply R3 to these three columns, thus shifting column $c$ to position $i+2$. Continue applying R3 as needed to shift the two columns $c$, that is, those columns which are originally in positions $i$ and $j$, towards each other until they are in consecutive positions. Then apply R2 to remove both these columns. 

By repeating the process described above as many times as needed, we end up with a flat braid in which each column features at most once. What flat braid is it? Recall that both the flat braid $[ES2]$ we started from and each of the flat braids we transformed it to satisfy the condition that the alternating sum of each row is $0$. There is only a small number of flat braids in which each column features at most once; they are listed below (in the ES2 encoding). 

$$
\begin{pmatrix}\; \\ \; \\ \; \end{pmatrix},
\begin{pmatrix}1\\1\\0\end{pmatrix},
\begin{pmatrix}0\\1\\1\end{pmatrix},
\begin{pmatrix}0 & 1\\1 & 0\\1 & 1\end{pmatrix},
\begin{pmatrix}1 & 1\\1 & 0\\0 & 1\end{pmatrix},
\begin{pmatrix}1 & 1 & 0\\1 & 0 & 1\\0 & 1 & 1\end{pmatrix} = \begin{pmatrix}0 & 1 & 1\\1 & 0 & 1\\1 & 1 & 0\end{pmatrix}.
$$

By inspecting all these flat braids, we can see that there is only one among them in which the alternating sum of each row is $0$; this is the empty braid. Thus, $[ES2]$ can be transformed into the empty braid using R2 and R3, therefore, $[ES2]$ is trivial. 
\end{proof}

From Theorem \ref{thm:ces_permutation_description} it follows that the list of alternating sums of rows is not only an invariant, but a \emph{complete} invariant of a flat braid, that is, two flat braids are equal as group elements if and only if they have the same list of alternating sums of rows. To say it explicitly, all $6$ possible values of the list of alternating sums of rows in $[ES2]$ representing a flat braid on $3$ strands are as follows (for convenience, we list them in the order corresponding to the list of flat braids above). 

$$
\begin{pmatrix}0\\0\\0\end{pmatrix},
\begin{pmatrix}1\\1\\0\end{pmatrix},
\begin{pmatrix}0\\1\\1\end{pmatrix},
\begin{pmatrix}-1\\1\\0\end{pmatrix},
\begin{pmatrix}0\\1\\-1\end{pmatrix},
\begin{pmatrix}0\\2\\0\end{pmatrix}.
$$
Based on this observation, we can formulate a simple algorithm for testing if a matrix $[ES2]$ is realizable. Let us denote by $[ES2]_m$ the matrix produced from $[ES2]$ by keeping only the leftmost $m$ columns and discarding other columns. If $[ES2]$ is realizable then each $[ES2]_m$ is realizable. If $[ES2]$ is not realizable then there is $m$ such that $[ES2]_{m-1}$ is realizable but $[ES2]_m$ is not realizable, and it is easy to check that the fact that $[ES2]_m$ is not realizable due to having a forbidden last column must be reflected in a ``faulty'' value of the invariant of $[ES2]_m$. Therefore, matrix $[ES2]$ is realizable if and only if for each $m = 1, \dots, k$ the list of alternating sums of $[ES2]_m$ is equal to one of the $6$ columns above.

In keeping with the spirit of this article, we have also checked if MLPs can detect realizable flat braids (in ES2 or ES1). In our experiment, MLPs learned surprisingly well to detect realizable flat braids, but we were not able to produce a meaningful interpretation of their weights, so we could not formulate a mathematical conjecture based on them.

We have also attempted to experiment with flat braids on $4$ strands (in other words, permutations on $4$ elements), and MLPs do not produce any conjectures which we can attempt to prove as theorems. To remind the reader, in flat braids on $4$ elements, apart from moves R2 and R3, one also is expected to use the move that replaces $2$ consecutive columns $\begin{pmatrix}1 & 0\\1 & 0\\0 & 1\\0 & 1\end{pmatrix}$ by $\begin{pmatrix}0 & 1\\0 & 1\\1 & 0\\1 & 0\end{pmatrix}$ or vice versa. Obviously, the alternating sums of rows are not preserved by this move. Thus, if there are invariants of flat braids on $4$ strands, they will have a form completely different to that in Theorem \ref{thm:ces_permutation_description}.

\section{Pure braids} \label{sec:pure}

In Theorem \ref{thm:ces_braid_description} condition CES is shown to be a necessary condition for a braid to be trivial, and in Theorem \ref{thm:ces_permutation_description} condition CES is shown to be equivalent to a flat braid being trivial. Recall that a braid is called pure if the flat braid corresponding to it is a trivial flat braid. If you compare Theorems \ref{thm:ces_braid_description} and \ref{thm:ces_permutation_description}, you might ask if a braid is pure if and only if it satisfies CES. We discuss and answer this question in this section. 

It is not true that every pure braid satisfies CES. Actually, a majority of pure braids do not satisfy CES; indeed, every sixth braid is pure, but only a tiny minority of braids satisfy CES. The simplest specific counterexample is $\sigma_1^2$.

It is not true that if a braid satisfies CES2 then it is pure. The braid $\sigma_1 \sigma_2 \sigma_1$, with $[ES2] = \begin{pmatrix} 1 & 1 & 0  \\ -1 & 0 & 1 \\ 0 & -1 & -1 \end{pmatrix}$ is a counterexample.

However, it is true that if a braid satisfies CES1 then it is pure. The proof of this little fact is unexpectedly involved; we prove it in the following theorem. 

\begin{theorem} \label{thm:pure}
If a braid satisfies condition CES1 then it is a pure braid.
\end{theorem} 
\begin{proof}
For a braid $[ES1]$, let us denote the corresponding flat braid by $[ESf1]$; that is, the difference between $[ESf1]$ and $[ES1]$ is that in $[ESf1]$ all entries which are equal to $-1$ in $[ES1]$ are replaced by $1$. 

The proof will proceed as an expanded version of the proof of Theorem \ref{thm:ces_permutation_description}. We will simplify the braid applying the same Reidemeister moves (that is, Reidemeister moves affecting the same columns) to $[ES1]$ and $[ESf1]$. The type of Reidemeister moves we use will be slightly wider than valid braid Reidemeister moves on $[ES1]$ and slightly narrower than valid flat braid Reidemeister moves on $[ESf1]$. Namely, we will use Reidemeister moves which, on the one hand, preserve the alternating sums of columns of $[ES1]$ and, on the other hand, enable us to simplify $[ESf1]$ to obtain the empty flat braid. By successfully reducing $[ESf1]$ to the empty flat braid we will demonstrate that $[ESf1]$ is a trivial flat braid and, equivalently, $[ES1]$ is pure.

The result is true for the empty braid with $k=0$. Now assume that $k \ge 1$. Then $[ESf1]$ contains a column, say, $c = \begin{pmatrix} 1  \\ 0    \\ 0 \end{pmatrix}$. Since the alternating sum of row $1$ in $[ES1]$ is $0$, there must be at least one other column in $[ESf1]$ which is also equal to $c$. Let us list all column positions of columns which are equal to $c$ in $[ESf1]$; they are $p_1, \dots, p_m$. By Lemma \ref{lem:flat}, the difference $p_{i+1} - p_i$ is odd for each $i$. What would happen if in $[ES1]$ every column at position $p_i$ for odd values of $i$ was $c$ and every column at position $p_i$ for even values of $i$ was $-c$; or what would happen if every column at position $p_i$ for odd values of $i$ was $-c$ and every column at position $p_i$ for even values of $i$ was $c$? Then the alternating sum of row $1$ in $[ES1]$ would be $m$ or $-m$ and not $0$. Therefore, there are an odd $i$ and an even $j$ such that columns in positions $p_i, p_j$ are both equal to $c$ or both equal to $-c$. Therefore, for some $i$, columns in positions $p_i, p_{i+1}$ are both equal to $c$ or both equal to $-c$. Use the move R3 to shift these columns towards each other until they are in consecutive positions; perform the Reidemeister moves synchronously in $[ES1]$ and $[ESf1]$, and ignore the condition in the definition of R3 for braids which does not allow one to use it in a braid when $x, y, z$ are equal to each other. Then apply R2 to remove the columns. Note that condition CES1 remains true for the modified $[ES1]$.

Applying the same process as many times as needed, we find a sequence of applications of R3 and R2 which transforms $[ESf1]$ into the empty flat braid.
\end{proof}

\begin{corollary}
CES1 is a necessary condition for a braid to be trivial and a sufficient condition for a braid to be pure.
\end{corollary}

\section{The gap between CES2 and CES1} \label{sec:gap-CES2-CES1}

The braid $\sigma_1 \sigma_2 \sigma_1$, which we saw in Section \ref{sec:pure}, has $[ES2] = \begin{pmatrix} 1 & 1 & 0  \\ -1 & 0 & 1 \\ 0 & -1 & -1 \end{pmatrix}$ and $[ES2] = \begin{pmatrix} 1 & 0 & 0  \\ 0 & 0 & 1 \\ 0 & -1 & 0 \end{pmatrix}$. Therefore, it is an example of a braid which satisfies condition CES2 but not CES1 (and not CEP). 

\begin{proposition} \label{prop:ES1-to-ES2}
If a braid satisfies condition CES1 then is satisfies condition CES2. \\
\end{proposition}
\begin{proof}
This follows from our observation that $[ES2] = \begin{pmatrix} 1 & 0 & -1  \\ -1 & 1 & 0 \\ 0 & -1 & 1 \end{pmatrix} [ES1]$.
\end{proof}

\begin{lemma} \label{lem:EP1-vs-ES1}
The sum of all entries in $[EP1]$ is equal to the sum of alternating sums of rows of $[ES1]$. In other words, $[1]_{1 \times 2}[EP1][1]_{k \times 1} = [1]_{1 \times 3}[ES1][\pm]_{k \times 1}$
\end{lemma}
\begin{proof}
Let us denote the columns that might feature in [ES1] by $c_1$, $-c_1$, $c_2$, $-c_2$, $c_3$, $-c_3$, where $c_1 = \begin{pmatrix} 1  \\ 0    \\ 0 \end{pmatrix}$, $c_2 = \begin{pmatrix} 0  \\ 1    \\ 0 \end{pmatrix}$, $c_3 = \begin{pmatrix} 0  \\ 0    \\ 1 \end{pmatrix}$. It is easy to prove by induction on the length of the braid that the leftmost column position in which $\pm c_1 (\pm c_2, \pm c_3)$ can feature has an odd (odd, even) position in the braid. As stated in Lemma \ref{lem:flat} and discussed in the proof of Theorem \ref{thm:pure}, each next column featuring $\pm c_1 (\pm c_2, \pm c_3)$ appears an odd number of positions after the previous column featuring $\pm c_1 (\pm c_2, \pm c_3)$. 

Thus, $c_1$ in an odd position contributes $1$ to the alternating sums of $[ES1]$ and, being a clockwise half-turn, also $1$ to $[EP1]$, whereas $c_1$ in an even position contributes $-1$ to the alternating sums of $[ES1]$ and, being an anticlockwise half-turn, also $-1$ to $[EP1]$. As to $-c_1$, $-c_1$ in an odd position contributes $-1$ to the alternating sums of $[ES1]$ and, being an anticlockwise half-turn, also $-1$ to $[EP1]$, whereas $-c_1$ in an even position contributes $1$ to the alternating sums of $[ES1]$ and, being a clockwise half-turn, also $-1$ to $[EP1]$. The same argument also applies to $c_2$ and $-c_2$. 

As to $c_3$ ($-c_3$), recall that it is strand $3$ crossing over (under) strand $1$. Thus, $-c_3$ in an even position (see Figure \ref{fig:braid-example} for an example where $-c_3$ features in an even position) contributes $1$ to the alternating sums of $[ES1]$ and, being a clockwise half-turn, also $1$ to $[EP1]$, whereas $-c_3$ in an odd position contributes $-1$ to the alternating sums of $[ES1]$ and, being an anticlockwise half-turn, also $-1$ to $[EP1]$. Likewise, $c_3$ in an even position contributes $-1$ to the alternating sums of $[ES1]$ and, being an anticlockwise half-turn, also $-1$ to $[EP1]$, whereas $c_3$ in an odd position contributes $1$ to the alternating sums of $[ES1]$ and, being a clockwise half-turn, also $1$ to $[EP1]$.

Hence, the sum of all entries in $[EP1]$ is equal to the sum of alternating sums of rows of $[ES1]$.
\end{proof}

\begin{corollary} \label{cor:CES1-CEP}
If a braid satisfies CES1 then it satisfies CEP.
\end{corollary}

\begin{theorem} \label{thm:CEP-CES2-CES1}
If a braid satisfies CEP and CES2 then it satisfies CES1. 
\end{theorem}
\begin{proof}
Suppose $[EP2]$ satisfies CEP. Considering the formula recalculating $[EP1]$ into $[EP2]$ presented in the proof of Proposition \ref{prop:ES1-to-ES2}, we conclude that the alternating sums of rows in $[EP1]$ are equal to each other, that is, have a form $\begin{pmatrix} a  \\ a    \\ a \end{pmatrix}$ for some number a. Hence, using Lemma \ref{lem:EP1-vs-ES1}, we conclude that the sum of entries in $[EP1]$ is equal to $3a$. Since condition CEP is satisfied, $3a=0$, hence, $a=0$, hence, the alternating sums of rows in $[EP1]$ are all equal to $0$.
\end{proof}

\section{Discovering theorems on braids on $2$ strands} \label{sec:2-strands}


The group $B_2$ of braids on $2$ strands is isomorphic to the infinite cyclic group, so we were certain that MLPs can easily learn to classify braids on $2$ strands as trivial or non-trivial both in ES and EP1.  As expected, we were able to train ML models to achieve 100\% precision in many settings. See Tables~\ref{table:6},\ref{table:7} for the results on machine learning using ES2 and EP1 encodings, respectively. 

\begin{table}[hbt!]
\begin{tabular}{||c c c c||} 
 \hline
 H & 50\% /2-fold  & 67\%/3-fold  & 75\% /4-fold\\ [0.5ex] 
 \hline\hline
 1 & 0.830 & 0.609 & 0.685 \\ 
 \hline
 2 & 1.000 & 0.723 & 0.881 \\
 \hline
 3 & 1.000 & 1.000 & 0.950 \\
 \hline
 4 & 1.000 & 1.000 & 1.000 \\
 \hline
 5 & 1.000 & 1.000 & 1.000 \\ 
 \hline
 a=13 & 1.000 & 1.000 & 1.000 \\ 
 \hline
\end{tabular}
\caption{(2,12)-braids recognition. Encoding ES2. Dataset of 1000 braids.    Weighted Avg. Precision for MLP with $H$ neurons in hidden layer and for k\% training set. 2-fold, 3-fold and 4-fold validation applied. 
}
\label{table:6}
\end{table}

\begin{table}[hbt!]
\begin{tabular}{||c c c c||} 
 \hline
 H & 50\% /2-fold  & 67\%/3-fold  & 75\% /4-fold\\ [0.5ex] 
 \hline\hline
 1 & 0.631 & 0.635 & 0.595 \\ 
 \hline
 2 & 0.820 & 1.000 & 0.881 \\
 \hline
 3 & 1.000 & 1.000 & 0.899 \\
 \hline
 4 & 1.000 & 1.000 & 1.000 \\
 \hline
 5 & 1.000 & 1.000 & 1.000 \\ 
 \hline
 a=7 & 1.000 & 1.000 & 1.000 \\ 
 \hline
\end{tabular}
\caption{(2,12)-braids recognition. Encoding EP1. Dataset of 1000 braids.    Weighted Avg. Precision for MLP with $H$ neurons in hidden layer and for k\% training set. 2-fold, 3-fold and 4-fold validation applied. 
}
\label{table:7}
\end{table}

\section{An application: An algorithm for finding trivial braids}

The authors conduct research in applying deep learning to mathematical reasoning, and untangling trivial braids using Reidemeister moves is one of our favorite problems in this area \cite{khan2021untangling,lisitsa2023supervised}. This is why we regularly generate datasets of random trivial braids to use for training or testing our AI agents. We use an implementation of an algorithm for checking triviality which we will denote by AUT; the details of AUT are of minor importance in this paper, but we will briefly present them in the next section\footnote{When we were writing the section, we searched for other implementations of triviality-checking algorithms and could not find any which we could run easily. This is why in this section we show how the results of this paper can be applied to our algorithm AUT.}. 

The following paragraphs summarize the experiments presented in more detail further in this section.

The easiest way to produce a dataset of uniformly sampled trivial braids is by generating random braids, one after another, and checking if they are trivial, until you have found as many trivial braids as you need. Since only a minority of braids are trivial, the task feels like finding the proverbial needle in a haystack. Theorems \ref{thm:ces_braid_description} and \ref{thm:cep_braid_description} provide new necessary conditions for a braid on $3$ strands to be trivial, and these simple conditions can be checked very fast. Therefore, Theorems \ref{thm:ces_braid_description} and \ref{thm:cep_braid_description} can be used to quickly discard a braid if it is non-trivial, before the slower algorithm AUT is applied in those cases when Theorems \ref{thm:ces_braid_description} and \ref{thm:cep_braid_description} do not give a definite answer.

Accuracy of Theorem \ref{thm:ces_braid_description} in detecting trivial braids is much higher than Theorem \ref{thm:cep_braid_description}; this is why the MLPs have found Theorem \ref{thm:ces_braid_description} before \ref{thm:cep_braid_description}. However, encoding ES used in the condition of Theorem \ref{thm:ces_braid_description} takes longer to construct than EP, so in terms of time efficiency, it is slightly better to check the condition CEP from Theorem \ref{thm:cep_braid_description} first, and only if it is satisfied, check the condition CES from Theorem \ref{thm:ces_braid_description}. This algorithm, which we will call CEP-CES-AUT, is faster than the combination CES-CEP-AUT which the MLPs seem to suggest.

Corollary \ref{cor:CES1-CEP} suggests that we can use the combination CES1-AUT, since CES1 implies CEP. However, discarding most non-trivial braids using condition CEP in combination CEP-CES-AUT is faster than building encoding ES1 for all braids and using combination CES1-AUT.





\subsection{Many braids can be classified using CES or CEP instead of AUT}

Recall that for each braid of length $k$, the majority of braids are non-trivial. As Table \ref{tab:count-braids-checked-by-CEP-CES-AUT} shows, most of the non-trivial braids can be classified as non-trivial using conditions CEP or CES2 or CES1. Looking at this table, it is useful to remember that we proved that $\mathrm{CES1} \Rightarrow \mathrm{CES2}$, $\mathrm{CES1} \Rightarrow \mathrm{CEP}$, $\mathrm{CEP \; \& \; CES2} \Leftrightarrow \mathrm{CES1}$, and the triviality implies all CEP, CES2 and CES1, see Theorem \ref{thm:ces_braid_description}, Theorem \ref{thm:cep_braid_description}, Proposition \ref{prop:ES1-to-ES2}, Corollary \ref{cor:CES1-CEP} and Theorem \ref{thm:CEP-CES2-CES1}.

For this experiment we generated 10,000 random braids of the same length on $3$ strands and checked conditions CEP, CES2, CES1, CEP \& CES2, and also triviality (using algorithm AUT) We repeated the experiment for lengths $10, \dots, 50$. 

\begin{table} \centering
\begin{tabular}{|c|c|c|c|c|c|c|}
  \hline
  Length & CEP & CES2 & CES1 & CEP \& CES2 & AUT & Time (s) \\
  \hline
    10 & 2475 & 379 & 301 & 301 & 262 & 1.76 \\ \hline
    20 & 1731 & 168 & 103 & 103 & 64 & 2.91 \\ \hline
    30 & 1530 & 110 & 47 & 47 & 15 & 5.93 \\ \hline
    40 & 1273 & 90 & 36 & 36 & 11 & 25.19 \\ \hline
    50 & 1102 & 69 & 24 & 24 & 5 & 147.68\\  
  \hline
\end{tabular}
\caption{How many braids, out of 10,000, satisfy conditions CEP, CES2, CES1, CEP \& CES2, and how many are trivial (checked by algorithm AUT). The column `Time' shows the total time in seconds.}\label{tab:count-braids-checked-by-CEP-CES-AUT}
\end{table}

\subsection{Using CES and CEP makes checking faster}

In this experiment we generate 10,000 random braids and then for each braid we check if condition CEP is satisfied. If CEP is satisfied, condition CES2 is checked for this braid. Then, if CES2 is also satisfied, the braid is further checked using AUT. In Table \ref{tab:time-braids-checked-by-CEP-CES-AUT}, the columns CEP (CES, AUT) refer to the number of braids that satisfy CEP (satisfy both CEP and CES, are found to be trivial by AUT). As you can see, only a minority of braids, about $0.2\%$--$3\%$ of braids need to be checked using AUT; indeed, we have already seen similar numbers in Table \ref{tab:count-braids-checked-by-CEP-CES-AUT}. The columns T CEP (T CES2, T AUT) correspond to the total time taken to check CEP (or CES2, or run AUT) for those braids for which these checks were conducted. The column TT contains the total time taken to check the triviality of the 10,000 braids. 

As you can see, the total time in Table \ref{tab:time-braids-checked-by-CEP-CES-AUT} is considerably smaller than the total time in Table \ref{tab:count-braids-checked-by-CEP-CES-AUT}, because AUT is only run for a minority of braids and not all braids. 

As you scan the columns of Table \ref{tab:time-braids-checked-by-CEP-CES-AUT}, you can see that the proportion of braids that need to be checked using AUT (column CES2) becomes smaller as the braid length $k$ grows, but the time taken to run AUT on these longer braids becomes larger (column T AUT).

\begin{table} \centering
\begin{tabular}{|c|c|c|c|c|c|c|c|c|c|}
    \hline
    Length & CEP & CES2 & AUT & T CEP & T CES2 & T AUT & TT \\
    \hline
    10 & 2456 & 279 & 248 & 0.029 & 0.057 & 0.011 & 0.097 \\ \hline
    20 & 1764 & 95 & 55 & 0.051 & 0.076 & 0.011 & 0.138 \\ \hline
    30 & 1456 & 54 & 20 & 0.073 & 0.092 & 0.022 & 0.186 \\ \hline
    40 & 1252 & 34 & 8 & 0.094 & 0.103 & 0.064 & 0.262 \\ \hline
    50 & 1122 & 24 & 3 & 0.116 & 0.115 & 0.741 & 0.972 \\ 
    \hline
\end{tabular}
\caption{The number of braids considered and time spent when checking triviality using combination CEP-CES-AUT}\label{tab:time-braids-checked-by-CEP-CES-AUT}
\end{table}

\subsection{CEP-CES-AUT is the fastest combination}

Theorem \ref{thm:ces_braid_description}, Theorem \ref{thm:cep_braid_description}, Proposition \ref{prop:ES1-to-ES2}, Corollary \ref{cor:CES1-CEP} and Theorem \ref{thm:CEP-CES2-CES1} suggest several sequences in which CEP, CES1, CES2 and AUT can be combined to check the triviality, including AUT on its own, CEP followed by AUT, CES1 followed by AUT, CES2 followed by CEP followed by AUT, and CEP followed by CES2 followed by AUT. In the graph in Figure \ref{fig:CEP-CES-AUT-is-fastest} we compare how fast these implementations are, and we see that combination CEP-CES2-AUT is the fastest. 

For this experiment, we use each of the algorithms from the list above, for each value of braid length $k$ in the range $10, 14, \dots, 34$, to produce a dataset consisting of $1000$ trivial braids and $1000$ non-trivial braids. 

The line corresponding to CES1-AUT is only marginally higher than the line corresponding to CES2-CEP-AUT. This is because CES1 is equivalent to CES2 \& CEP, the time needed for producing encoding ES1 is practically the same (just slightly longer, in our implementation) as the time needed for producing encoding ES2, and hardly any braids need to be checked for CEP after CES2 has been checked.

\begin{figure}
  \centering
  \begin{subfigure}{0.49\textwidth}
    \includegraphics[width=\linewidth]{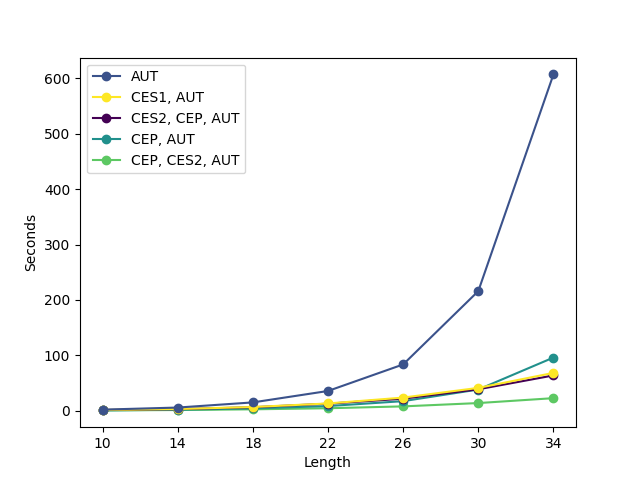}
    \caption{Linear scale}
    \label{fig:exp_3_lin}
  \end{subfigure}
  \begin{subfigure}{0.48\textwidth}
    \includegraphics[width=\linewidth]{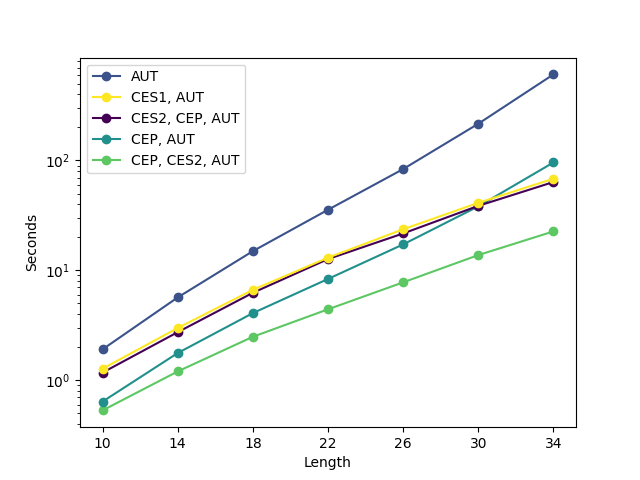}
    \caption{Log scale}
    \label{fig:exp_3_log}
  \end{subfigure}
  \caption{Out of several ways of checking triviality, CEP-CES2-AUT is the fastest.}
  \label{fig:CEP-CES-AUT-is-fastest}
\end{figure}


\subsection{Comparing AUT and CEP-CES-AUT on individual braids}

In this experiment we generated 10,000 random braids on length $20$ on $3$ strands. We checked condition AUT on every braid and recorded the time taken by AUT on this individual braid. These times are presented by the green bar chart in Figure \ref{fig:bar-charts}, with the time (the horizontal axis) on the log scale; to present the data more clearly, we show the same bar chart two times on different scales, with a smaller size in (a) and with a magnified size in (b). On the other hand, we tested the triviality by checking conditions CEP, CES, AUT on the same braids and recorded the time for each braid. These times are presented on the same bar chart in blue. 

Thus, the green chart shows a typical time it takes to check if a braid is trivial using AUT, which is about $10^{-4}$. The blue chart shows a typical time it takes to check if a braid is trivial using CEP-CES-AUT. If you look at part (a), the high blue spike on the left-hand side, that is, at about $5\cdot10^{-6}$ on the horizontal coordinate, correspond to braids classified by CEP. The next blue spike, at about $5\cdot10^{-5}$, corresponds to braids classified by CES. Then the small blue hump at about $10^{-4}$ (it is easier to spot it in part (b)) corresponds to the braids which CEP and CES could not classify; to this minority of braids, AUT is applied. 

\begin{figure}
  \centering
  \begin{subfigure}{0.49\textwidth}
    \includegraphics[width=\linewidth]{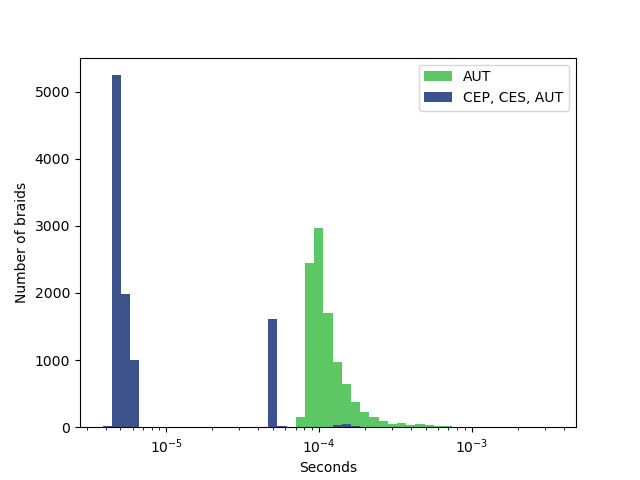}
    \caption{A plot showing all the data}
    \label{fig:exp_4_full}
  \end{subfigure}
  \begin{subfigure}{0.48\textwidth}
    \includegraphics[width=\linewidth]{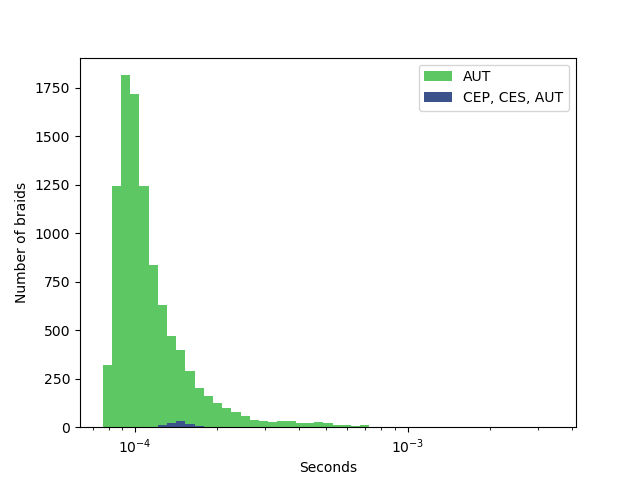}
    \caption{A magnified fragment of (a)}
    \label{fig:exp_4_zoom}
  \end{subfigure}
  \caption{Time to check the triviality taken by AUT (in green) and by CEP-CES-AUT (in blue).}
  \label{fig:bar-charts}
\end{figure}

\subsection{How much search it takes to find trivial braids}

In this experiment we produce a dataset consisting of $1,000$ pairwise distinct trivial braids. The aim of this exercise is to show how many braids one had to consider to find $1,000$ trivial braids, and how well our new algorithm CEP-CES-AUT copes with this task in terms of time. Table \ref{tab:produce-dataset} shows how many braids we had to explore in total (column EXP) before we found $1,000$ trivial braids, how many of these braids satisfied condition CEP (column CEP), how many of these braids satisfied both CEP and CES and were then checked using AUT (column CES), and how many of these braids were found trivial by AUT (column AUT). For example, in order to produce $1,000$ pairwise distinct trivial braids of length $10$ we had to consider $42,999$ random braids, out of which $10,515$ satisfied CEP, $1,126$ satisfied CEP and CES, and $1,021$ were found trivial by AUT; out of these $1,021$ braids, $21$ braids were discarded because their duplicates had already been added to the dataset. 

\begin{table} \centering
\begin{tabular}{|c|c|c|c|c|c|c|c|c|c|}
    \hline
    Length & EXP & CEP & CES & AUT & Time \\
    \hline
    10 & 42,999 & 10,515 & 1,126 & 1,021 & 0.574 \\ \hline
    20 & 167,694 & 29,596 & 1,677 & 1,000 & 3.306 \\ \hline
    30 & 515,663 & 74,347 & 2,737 & 1,000 & 13.689 \\ \hline
    40 & 1,241,628 & 155,568 & 4,394 & 1,000 & 43.899 \\ \hline
    50 & 2,796,519 & 313,277 & 6,941 & 1,000 & 155.987 \\ \hline
\end{tabular}
\caption{How many braids need to be explored to find $1,000$ trivial braids}\label{tab:produce-dataset}
\end{table}

\section{Details of algorithm AUT} \label{sec:AUT}

Algorithm AUT is based on Artin’s solution of the word problem in the braid group via automorphisms of a free group (this is why we denote this algorithm by AUT). Algorithm AUT is exponential in the worst case, but it is fast in practice; for example, we used AUT to produce all words representing trivial braids on $3$ strands up to length $k=16$, see https://oeis.org/A354602. We are aware that there are algorithms for checking if a braid is trivial which have a better algorithmic complexity, polynomial instead of exponential, for example \cite{birman1998new},  but in our research practice, AUT works sufficiently fast, and in the context of this study, our conclusions regarding combining AUT with CEP and CES can be applied to any other algorithm instead of AUT.

The differences between our algorithm AUT and a naive implementation of Artin's solution of the word problem are as follows. 

We do not use Artin's automorphisms of the free group, but automorphisms introduced in Theorem 1 (automorphism type 2) in \cite{shpilrain2001representing}. The benefit of using these automorphisms is that in each word considered, each letter at an odd-numbered position is one of generators of the free group, and each letter at an even-numbered position is the inverse of one of generators  of the free group; this useful property enables us to write especially simple computer code for simplifying words to reduced words. The way this type of automorphisms is used in AUT can be further improved using algebraic structures known as \emph{keis} or \emph{involutory quandles}. In the past, the authors of this paper studied keis \cite{fish2014detecting,fish2015combinatorial,lisitsa2017automated,vernitski2018dihedral}; keis can be informally described as groups considered with the operation $x \triangleright y = y x^{-1} y$ instead of the usual group operation. The automorphisms from \cite{shpilrain2001representing} can be usefully represented as kei terms instead of group words, and these kei terms are considerably shorter than the corresponding group words (although they are still exponentially long in the worst case; the length of the kei terms grows, in the worst case, as Fibonacci numbers instead of $2^k$ for group words). 

Instead of using kei terms, we obtain approximately the same effect in AUT by judiciously using dichotomy to ensure that we get an opportunity to simplify group words as soon whenever possible, before they become very long. Namely, when dealing with a word $b$ in the braid group, we decompose it into a product of words $b=cd$, where $c, d$ have an equal (or almost equal) length, compute the automorphisms induced by $c$ and $d$, inspect the description of these automorphisms and simplify the words in the free group to reduced words, and then multiply the automorphisms corresponding to $c$ and $d$ to produce the automorphism corresponding to $b$.

\section*{Acknowledgements}
The work was supported by the Leverhulme Trust Research Project Grant RPG-2019-313.

\bibliographystyle{abbrv}
\bibliography{main}
\section*{Appendix}
 Default settings of MLP in  WEKA Workbench (version 3.8.3): 

\verb"L = 0.3" \emph{(learning rate)}

\verb"M = 0.2" \emph{(momentum rate)} 

\verb"N = 500" \emph{(number of epochs to train)} 

\verb"V = 0" \emph{(percentage size of validation set to use to terminate training)}

\verb"S = 0" \emph{(seed for Random Number Generator)}

\verb"E = 20" \emph{(threshold for number of consecutive errors to use to terminate training)}

\end{document}